\documentclass{article}
\usepackage{indentfirst}
\usepackage[english]{babel}
\usepackage{amsmath, amssymb}
\usepackage{amsthm}
\usepackage{amsfonts}
\usepackage{enumerate}
\usepackage[latin1]{inputenc}
\usepackage{graphics}
\usepackage{comment}

\allowdisplaybreaks[4]

\newtheorem{thm}{Theorem}

\newtheorem{lemm}[thm]{Lemma}

\newtheorem{rem}[thm]{Remark}

\begin{document}

\title{Analyticity and criticality results for the eigenvalues of the biharmonic operator}
\author{Davide Buoso\footnote{Politecnico di Torino. Email: davide.buoso@polito.it.}}
\date{ }
\maketitle

\noindent
{\bf Abstract:}
We consider the eigenvalues of the biharmonic operator subject to several homogeneous boundary conditions (Dirichlet, Neumann, Navier, Steklov). We show that simple eigenvalues and elementary symmetric functions of multiple eigenvalues are real analytic, and provide Hadamard-type formulas for the corresponding shape derivatives. After recalling the known results in shape optimization, we prove that balls are always critical domains under volume constraint.

\vspace{11pt}


\vspace{6pt}

\section{Introduction}
\label{intro}

In this paper we consider eigenvalue problems for the biharmonic operator subject to several homogeneous boundary conditions in bounded domains $\Omega$ in $\mathbb R^N$, $N\ge2$. Note that such problems arise in the study of vibrating plates within the so-called Kirchhoff-Love model (see e.g., \cite{rayleigh}). In particular, we consider the following equation
\begin{equation}
\label{biha}
\Delta^2u-\tau\Delta u=\lambda u,{\rm\ in\ }\Omega,
\end{equation}
where $\tau$ is a non-negative constant related to the lateral tension of the plate. As for the boundary conditions, we are interested in Dirichlet boundary conditions
	\begin{equation}
	\label{dir}
	u=\frac{\partial u}{\partial\nu}=0\ \mathrm{on\ }\partial\Omega,
	\end{equation}
which are related to clamped plates, Navier boundary conditions
\begin{equation}
\label{nav}
u=(1-\sigma)\frac{\partial^2u}{\partial\nu^2}+\sigma\Delta u=0\ \mathrm{on\ }\partial\Omega,
\end{equation}
which are related to hinged plates, and Neumann boundary conditions
	\begin{equation}
	\label{neu}
	(1-\sigma)\frac{\partial^2u}{\partial\nu^2}+\sigma\Delta u=\tau\frac{\partial u}{\partial\nu}-\frac{\partial\Delta u}{\partial\nu}-
		(1-\sigma)	\mathrm{div}_{\partial\Omega}\left(\nu^tD^2u\right)_{\partial\Omega}=0\ \mathrm{on\ }\partial\Omega,
		\end{equation}
which are related to free plates. Note that $\sigma$ denotes the Poisson ratio of the material, typically $0\le\sigma\le0.5$. We recall that the conditions (\ref{neu}) have been known for long time only in dimension $N=2$ (see e.g., \cite{gine, verchota}), while the general case first appeared in \cite{chasman11} (see also \cite{chasman15}). We recall here that, given a vector function $f$, its tangential component is defined as $f_{\partial\Omega}=f-(f\cdot\nu)\nu$, and the tangential divergence operator is $\mathrm{div}_{\partial\Omega}f=\mathrm{div}f-\frac{\partial f}{\partial\nu}\cdot\nu$. 

We also consider Steklov-type problems for the biharmonic operator. Note that the first one to appear was the following
\begin{equation}
\label{ks}
\left\{\begin{array}{ll}
	\Delta^2u=0, & \text{in\ }\Omega,\\
	u=0, & \text{on\ }\partial\Omega,\\
\Delta u=\lambda \frac{\partial u}{\partial\nu}, & \text{on\ }\partial\Omega,
	\end{array}\right.
\end{equation}
and it was introduced in \cite{kuttler68}. Problem (\ref{ks}) has proved itself to be quite strange with respect to other Laplacian-related eigenvalue problem, at least concerning shape optimization results. In fact, differently from the classical Steklov problem where the interesting problem is the maximization of eigenvalues under volume constraint, here one searches for minimizers and, strikingly, the ball is not the optimal shape for the first eigenvalue (at least in dimension $N=2$), as shown in \cite{kuttler72}. Nevertheless, in \cite{buga11} the authors can prove that, among all convex domain of fixed measure there exists a minimizer, but nothing is known about the optimal shape, or if the convexity assumption can be relaxed. We also refer to \cite{antunes, begami, bufega} for other results on problem (\ref{ks}).

Another Steklov problem for the biharmonic operator which has appeared very recently in \cite{bupro} (see also \cite{buproimse}) is the following
\begin{equation}
\label{bp}
\left\{\begin{array}{ll}
	\Delta^2u-\tau\Delta u=0, & \text{in\ }\Omega,\\
	\frac{\partial^2u}{\partial\nu^2}=0,& \text{on\ }\partial\Omega,\\
	\tau\frac{\partial u}{\partial\nu}-\frac{\partial\Delta u}{\partial\nu}-\mathrm{div}_{\partial\Omega}\left(\nu^tD^2u\right)_{\partial\Omega}=\lambda u,&\mathrm{on\ }\partial\Omega.
	\end{array}\right.
	\end{equation}
In contrast with problem (\ref{ks}), problem (\ref{bp}) presents several spectral features resembling those of the Steklov-Laplacian. As shown in \cite{bupro}, problem (\ref{bp}) can be viewed as a limiting Neumann problem via mass concentration arguments (cf.\ \cite{lapro}), and moreover, for any fixed $\tau>0$, the maximizer of the first positive eigenvalue among all bounded smooth domains is the ball.
	
In this paper we are interested in analyticity properties of the eigenvalues of problems (\ref{biha})-(\ref{bp}). This type of analysis was first done by Lamberti and Lanza de Cristoforis in \cite{lala2004}, where they study regularity properties of the elementary symmetric functions of the eigenvalues of the Laplace operator subject to Dirichlet boundary conditions. Note that in general, when dealing with eigenvalues splitting from a multiple eigenvalue, bifurcation phenomena may occur, and the use of symmetric functions of the eigenvalues permits to bypass such situations. The techniques in \cite{lala2004} were later used to treat other types of boundary conditions (see \cite{lambertisteklov,lala2007}) and even other operators (see \cite{buosoimse, bulareis,bulaproc}). As for the biharmonic operator, this kind of analysis has been already carried out in several specific cases, see \cite{bula2013,buosohinged,bupro,buproimse}. Our aim here is to treat those cases altogether in order to give a general overview.

After proving that the elementary symmetric functions of the eigenvalues are analytic upon domain perturbations, we compute their shape differential. Following the lines of \cite{lalacri}, by means of the Lagrange Multiplier Theorem, we can  show that the ball is a critical domain under volume constraint for any of the elementary symmetric functions of the eigenvalues of problems (\ref{biha})-(\ref{bp}). We observe that, regarding problem (\ref{ks}), such a result was already obtained in \cite{buga11} but only for the first eigenvalue. We remark that the question of criticality of domains is strictly related with shape optimization problems, where the minimizing (resp.\ maximizing)  domain has to be found in a class of fixed volume ones. This type of problems for the eigenvalues of the biharmonic operator have been solved only in very specific cases, the optimal domain for the first eigenvalue being the ball (see \cite{ash,bupro,chasman11,chasman15,nadir}). As we have said above, for problem (\ref{ks}) the ball has been proved not to be the minimizer, nevertheless it still is a critical domain (cf.\ Theorem \ref{lepallesys}).

The paper is organized as follows. Section \ref{prel} is devoted to some preliminaries. In Section \ref{analitico} we examine the problem of shape differentiability of the eigenvalues. We consider problem (\ref{plateweak}) in $\phi(\Omega)$ and pull it back to $\Omega$, where $\phi$ belongs to a suitable class of diffeomorphisms. We also derive Hadamard-type formulas for the elementary symmetric functions of the eigenvalues. In Section \ref{critico} we consider the problem of finding critical points for such functions under volume constraint. We provide a characterization for the critical domains, and show that, for all the problems considered, balls are critical domains for all the elementary symmetric functions of the eigenvalues. Finally, in Section \ref{tech} we prove some technical results.


\section{Preliminaries}
\label{prel}

Let $N\in\mathbb{N}$, $N\ge2$, and let $\Omega$ be a bounded open set in $\mathbb{R}^N$ of class $C^1$. By $H^k(\Omega)$, $k\in\mathbb N$, we denote the Sobolev space of functions in $L^2(\Omega)$ with derivatives up to order $k$ in $L^2(\Omega)$, and by $H^k_0(\Omega)$ we denote the closure in $H^k(\Omega)$ of the space of $C^{\infty}$-functions with compact support in $\Omega$.

Let also $\tau\ge0$, $-\frac 1 {N-1}<\sigma<1$. We consider the following bilinear form on $H^2(\Omega)$
\begin{equation}
\label{prodotto}
P=(1-\sigma)M+\sigma B+\tau L,
\end{equation}
where
$$
M[u][v]=\int_{\Omega}D^2u:D^2v dx,\ \ B[u][v]=\int_{\Omega}\Delta u\Delta vdx,
$$
and
$$
L[u][v]=\int_{\Omega}\nabla u\cdot\nabla vdx,
$$
for any $u,v\in H^2(\Omega)$, where we denote by $D^2u:D^2v$ the Frobenius product $D^2u:D^2v=\sum_{i,j=1}^N\frac{\partial^2u}{\partial x_i\partial x_j}\frac{\partial^2v}{\partial x_i\partial x_j}$. We also consider the following bilinear forms on $H^2(\Omega)$
$$
J_1[u][v]=\int_{\Omega}uv dx,\ \ J_2[u][v]=\int_{\partial\Omega}\frac{\partial u}{\partial\nu}\frac{\partial v}{\partial\nu}d\sigma,\ \ J_3[u][v]=\int_{\partial\Omega}uvd\sigma,
$$
for any $u,v\in H^2(\Omega)$, where we denote by $\nu$ the unit outer normal vector to $\partial\Omega$, and by $d\sigma$ the area element.

Using this notation, problems \eqref{biha}-\eqref{neu} can be stated in the following weak form
$$
P[u][v]=\lambda J_1[u][v],\ \forall v\in V(\Omega),
$$
where $V(\Omega)$ is either $H^2_0(\Omega)$ (for the Dirichlet problem), or $H^2(\Omega)\cap H^1_0(\Omega)$ (for the Navier problem), or $H^2(\Omega)$ (for the Neumann problem). Here and in the sequel the bilinear forms defined on $V(\Omega)$ will be understood also as linear operators acting from $V(\Omega)$ to its dual.

As for Steklov-type problems, we shall consider their generalizations according to the definition of $P$. In particular, regarding problem \eqref{ks}, we consider the following generalization
\begin{equation}
\label{ksgen}
\left\{\begin{array}{ll}
	\Delta^2u-\tau \Delta u=0, & \text{in\ }\Omega,\\
	u=0, & \text{on\ }\partial\Omega,\\
(1-\sigma)\frac{\partial^2u}{\partial\nu^2}+\sigma\Delta u=\lambda \frac{\partial u}{\partial\nu}, & \text{on\ }\partial\Omega,
	\end{array}\right.
\end{equation}
whose weak formulation is
$$
P[u][v]=\lambda J_2[u][v],\ \forall v\in H^2(\Omega)\cap H^1_0(\Omega).
$$
We also consider the following generalization of problem \eqref{bp}
\begin{equation}
\label{bpgen}
\left\{\begin{array}{ll}
	\Delta^2u-\tau\Delta u=0, & \text{in\ }\Omega,\\
	(1-\sigma)\frac{\partial^2u}{\partial\nu^2}+\sigma\Delta u=0,& \text{on\ }\partial\Omega,\\
	\tau\frac{\partial u}{\partial\nu}-\frac{\partial\Delta u}{\partial\nu}-(1-\sigma)\mathrm{div}_{\partial\Omega}\left(\nu^tD^2u\right)_{\partial\Omega}=\lambda u,&\mathrm{on\ }\partial\Omega,
	\end{array}\right.
	\end{equation}
whose weak formulation is
$$
P[u][v]=\lambda J_3[u][v],\ \forall v\in H^2(\Omega).
$$

Using a unified notation, we can therefore write all the problems we are considering as
\begin{equation}
\label{plateweak}
P[u][v]=\lambda J_i[u][v],\ \forall v\in V(\Omega),
\end{equation}
where:
\begin{itemize}
\item for the Dirichlet problem \eqref{biha}, \eqref{dir} we set $i=1$, $V(\Omega)=H^2_0(\Omega)$;
\item for the Navier problem \eqref{biha}, \eqref{nav} we set $i=1$, $V(\Omega)=H^2(\Omega)\cap H^1_0(\Omega)$;
\item for the Neumann problem \eqref{biha}, \eqref{neu} we set $i=1$, $V(\Omega)=H^2(\Omega)$;
\item for the Steklov problem \eqref{ksgen} we set $i=2$, $V(\Omega)=H^2(\Omega)\cap H^1_0(\Omega)$;
\item for the Steklov problem \eqref{bpgen} we set $i=3$, $V(\Omega)=H^2(\Omega)$.
\end{itemize}

It is clear that both the Neumann problem \eqref{biha}, \eqref{neu} and the Steklov problem \eqref{bpgen} have non-trivial kernel. In particular, if $\tau>0$ then both kernels are given by the constant functions, while if $\tau=0$ the kernels have dimension $N+1$ including also the coordinate functions $x_1,\dots,x_N$ (cf.\ \cite[Theorem 3.8]{bupro}). For this reason, we will restrict our attention to the case $\tau>0$ and consider instead $V(\Omega)=H^2(\Omega)/\mathbb R$ for problems \eqref{biha}, \eqref{neu} and \eqref{bpgen} (the case $\tau=0$ being similar).

With this choice of the space $V(\Omega)$, it is possible to show that the bilinear form $P$ defines a scalar product on $V(\Omega)$ which is equivalent to the standard one. We shall therefore consider $V(\Omega)$ as endowed with such a scalar product.

It is easily seen that $P$, considered as an operator acting from $V(\Omega)$ to its dual, is a linear homeomorphism. In particular, we can define
\begin{equation}
\label{inverse}
T_i=P^{(-1)}\circ J_i,
\end{equation}
for $i=1,2,3$. We have the following

\begin{thm}
Let $-\frac{1}{N-1}<\sigma<1$, $\tau>0$. Let $\Omega$ be a bounded domain in $\mathbb{R}^N$ of class $C^1$. The operator\ $T_i$ defined in \eqref{inverse} is a non-negative compact selfadjoint operator on the Hilbert space $V(\Omega)$. Its spectrum is discrete and consists of a decreasing sequence of positive eigenvalues of finite multiplicity converging to zero. Moreover, the equation $T_iu=\mu u$ is satisfied for some $u\in V(\Omega)$, $\mu>0$  if and only if equation (\ref{plateweak})  is satisfied with $0\neq\lambda=\mu^{-1}$ for any $v\in V(\Omega)$. In particular, the eigenvalues of problem \eqref{plateweak} can be arranged in a diverging sequence
$$
0<\lambda_1[\Omega]\le\lambda_2[\Omega]\le\cdots\le\lambda_k[\Omega]\le\cdots,
$$
where all the eigenvalues are repeated according to their multiplicity, and the following variational characterization holds
$$
\lambda_k[\Omega]=\min_{\substack{E\le V(\Omega) \\ {\rm dim\ }E=k}}\max_{\substack{u\in E \\ J_i[u][u]\neq0}}\frac{P[u][u]}{J_i[u][u]}.
$$
\end{thm}

\begin{proof}
For the selfadjointness, it suffices to observe that
$$
< T_i[u],v>=<P^{(-1)}\circ J_i [u],v>=P[P^{(-1)}\circ J_i [u]][v]=J_i[u][v],
$$
for any $u,v\in V(\Omega)$. For the compactness, just observe that the operator $J_i$ is compact. The remaining statements are straightforward.
\end{proof}

\begin{rem}
As we have said in Section \ref{intro}, in applications $0\le\sigma\le0.5$. However, there are examples of materials with high or negative Poisson ratio, namely ($N=2$)
$$
-1<\sigma<1.
$$
In general dimension we choose $-\frac 1 {N-1}<\sigma<1$. This is due to the fact that, thanks to the inequality
$$
|D^2u|^2\ge\frac 1 N (\Delta u)^2\ \forall u\in H^2(\Omega),
$$
then, for $\sigma$ in that range, the operator $P$ turns out to be coercive. We also remark that, following the arguments in \cite[Section 3]{bupro}, the Steklov problem \eqref{bpgen} can be seen as a limiting Neumann problem of the type \eqref{biha}, \eqref{neu} with a mass distribution concentrating to the boundary.
\end{rem}

We note that problem \eqref{ks} is obtained for $\sigma=1$, which is out of our range. Under some additional regularity assumptions, for instance $\Omega\in C^2$ (see e.g., \cite{buga11} for general conditions), then the operator becomes coercive and all the results here and in the sequel apply as well. The same remains true also for the Navier problem \eqref{biha}, \eqref{nav}, which for $\sigma=1$ reads
\begin{equation*}
\left\{\begin{array}{ll}
\Delta^2u-\tau\Delta u=\lambda u, & {\rm in\ }\Omega,\\
u=\Delta u=0, & {\rm on\ }\partial\Omega,
\end{array}\right.
\end{equation*}
which has been extensively studied in the case $\tau=0$ (we refer to \cite{babu, ggs, parini, sweers} and the references therein).

The situation is instead completely different in the case of Neumann boundary conditions with $\sigma=1$, namely
\begin{equation}\label{badneu}
\left\{\begin{array}{ll}
	\Delta^2u-\tau \Delta u=\lambda u, & \text{in\ }\Omega,\\
	\Delta u=\frac{\partial\Delta u}{\partial\nu}=0, & \text{on\ }\partial\Omega.
\end{array}\right.
\end{equation}
It is easy to see that problem \eqref{badneu} has an infinite dimensional kernel since all harmonic functions belong to the eigenspace associated with the eigenvalue $\lambda=0$. In particular, the boundary conditions do not satisfy the complementing conditions, see \cite{adn,ggs}. We refer to \cite{prozkala} for considerations on the spectrum of problem \eqref{badneu}.


\section{Analyticity of the eigenvalues and Hadamard formulas}
\label{analitico}

The study of the dependence of the eigenvalues of elliptic operators on the domain has nowadays become a classical problem in the field of perturbation theory. Shape continuity of the eigenvalues has been known for long time (\cite{cohi}), and can also be improved to H\"older or Lipschitz continuity using stability estimates as in \cite{buosotesi,buda,bula,bulahigh,bulahighsh}. However, while the continuity holds for all the eigenvalues, only the simple ones enjoy an analytic dependence (see e.g., \cite{henry}). On the other hand, when the eigenvalue is multiple, bifurcation phenomena occur, so that, if the perturbation is parametrized by one real variable, then the eigenvalues are described by suitable analytic branches (cf.\ \cite[Theorem 1]{rellich}). Unfortunately, if the family of perturbations is not parametrized by one real variable, one cannot expect the eigenvalues to split into analytic branches anymore. In this case, the use of elementary symmetric functions of the eigenvalues (see \cite{lala2004,lala2007}) has the advantage of bypassing splitting phenomena, in fact such functions turn out to be analytic.

To this end, we shall consider problem (\ref{plateweak}) in a family of open sets parametri\-zed by suitable diffeomorphisms $\phi $
defined on a bounded open set $\Omega $ in ${\mathbb{R}}^N$ of class $C^1$. Namely, we set
$$
{\mathcal{A}}_{\Omega }=\biggl\{\phi\in C^2(\overline\Omega\, ; {\mathbb{R}}^N ):\ \inf_{\substack{x_1,x_2\in \overline\Omega \\ x_1\ne x_2}}\frac{|\phi(x_1)-\phi(x_2)|}{|x_1-x_2|}>0 \biggr\},
$$
where $C^2(\overline\Omega\, ; {\mathbb{R}}^N )$ denotes the space of all functions from $\overline\Omega $ to ${\mathbb{R}}^N$ of class $C^2$.  Note that if $\phi \in {\mathcal{A}}_{\Omega }$ then $\phi $ is injective, Lipschitz continuous and $\inf_{\overline\Omega }|{\rm det }\nabla \phi |>0$. Moreover, $\phi (\Omega )$ is a bounded open set of class $C^1$ and the inverse map $\phi^{(-1)}$ belongs to  ${\mathcal{A}}_{\phi(\Omega )}$.  Thus it is natural to consider problem (\ref{plateweak}) on $\phi (\Omega )$ and study  the dependence of $\lambda_k[\phi (\Omega )]$ on $\phi \in {\mathcal{A}}_{\Omega }$. To do so, we endow the space $C^2(\overline\Omega\, ; {\mathbb{R}}^N )$ with its usual norm.
Note that ${\mathcal{A}}_{\Omega }$ is an open set in  $C^2(\overline\Omega\, ;{\mathbb{R}}^N )$, see \cite[Lemma~3.11]{lala2004}.
    Thus, it makes sense to  study differentiability and analyticity properties of the maps $\phi \mapsto \lambda_k[\phi (\Omega )]$ defined for $\phi \in {\mathcal{A}}_{\Omega }$.
   For simplicity, we write  $\lambda_k[\phi ]$ instead of $\lambda_k[\phi (\Omega )]$.
   We fix a finite set of indexes $F\subset \mathbb{N}$
and we consider those maps $\phi\in {\mathcal{A}}_{\Omega }$ for which the eigenvalues
with indexes in $F$  do not coincide with eigenvalues with indexes not
in $F$; namely we set
$$
{\mathcal { A}}_{F, \Omega }= \left\{\phi \in {\mathcal { A}}_{\Omega }:\
\lambda_k[\phi ]\ne \lambda_l[\phi],\ \forall\  k\in F,\,   l\in \mathbb{N}\setminus F
\right\}.
$$
It is also convenient to consider those maps $\phi \in {\mathcal { A}}_{F, \Omega } $ such that all the eigenvalues with index in $F$
 coincide and set
$$
\Theta_{F, \Omega } = \left\{\phi\in {\mathcal { A}}_{F, \Omega }:\ \lambda_{k_1}[\phi ]=\lambda_{k_2}[\phi ],\, \
\forall\ k_1,k_2\in F  \right\} .
$$

 For $\phi \in {\mathcal { A}}_{F, \Omega }$, the elementary symmetric functions of the eigenvalues with index in $F$ are defined by
\begin{equation*}
\Lambda_{F,s}[\phi ]=\sum_{ \substack{ k_1,\dots ,k_s\in F\\ k_1<\dots <k_s} }
\lambda_{k_1}[\phi ]\cdots \lambda_{k_s}[\phi ],\ \ \ s=1,\dots , |F|.
\end{equation*}

We have the following 

\begin{thm}
\label{duesettesys}
Let $\Omega $ be a bounded open set in ${\mathbb{R}}^N$ of class $C^1$, $\tau>0$, $-\frac1{N-1}<\sigma<1$, and  $F$ be a finite set in  ${\mathbb{N}}$. 
The set ${\mathcal { A}}_{F, \Omega }$ is open in
$\mathcal{A}_{\Omega}$, and the real-valued maps 
$ \Lambda_{F,s}$ are real-analytic on  ${\mathcal { A}}_{F, \Omega }$, for all $s=1,\dots , |F|$. 
Moreover, if $\tilde \phi\in \Theta_{F, \Omega }  $ is such that the eigenvalues $\lambda_k[\tilde \phi]$ assume the common value $\lambda_F[\tilde \phi ]$ for all $k\in F$, and  $\tilde \phi (\Omega )$ is of class $C^{4}$ then  the Fr\'{e}chet differential of the map $\Lambda_{F,s}$ at the point $\tilde\phi $ is delivered by the formula
\begin{equation}
		\label{derivdsys}
		d|_{\phi=\tilde{\phi}}(\Lambda_{F,s})[\psi]
			=	\lambda_F^s[\tilde{\phi}]\binom{|F|-1}{s-1}
			\sum_{l=1}^{|F|}
			\int_{\partial\tilde{\phi}(\Omega)} G(v_l) (\psi\circ\tilde{\phi}^{(-1)})\cdot\nu d\sigma,
			\end{equation}
for all $\psi\in C^1(\overline\Omega;\mathbb{R}^N)$, where $\{v_l\}_{l\in F}$ is an orthonormal basis in $V(\tilde \phi (\Omega ))$ of the eigenspace associated with $\lambda_F[\tilde \phi]$ (the orthonormality being taken with respect to \eqref{prodotto}), and:
\begin{itemize}
\item $G(v)=-\left(\frac{\partial^2v}{\partial\nu^2}\right)^2$ for the Dirichlet problem;

\item $G(v)=(1-\sigma)|D^2v|^2+\sigma(\Delta v)^2+\tau|\nabla v|^2-\lambda_F[\tilde\phi(\Omega)]v^2$ for the Neumann problem;

\item $G(v)=2\frac{\partial v}{\partial\nu}\left(\frac{\partial\Delta v}{\partial\nu}+(1-\sigma){\rm div}_{\partial\tilde\phi(\Omega)}(\nu\cdot D^2v)_{\partial\tilde\phi(\Omega)}\right)+(1-\sigma)|D^2v|^2+\sigma(\Delta v)^2+\tau\left(\frac{\partial v}{\partial\nu}\right)^2$ for the Navier problem;

\item $G(v)=2\frac{\partial v}{\partial\nu}\left(\frac{\partial\Delta v}{\partial\nu}+(1-\sigma){\rm div}_{\partial\tilde\phi(\Omega)}(\nu\cdot D^2v)_{\partial\tilde\phi(\Omega)}\right)+(1-\sigma)|D^2v|^2+\sigma(\Delta v)^2+\tau\left(\frac{\partial v}{\partial\nu}\right)^2-\lambda_F[\tilde\phi(\Omega)]K\left(\frac{\partial v}{\partial\nu}\right)^2-\lambda_F[\tilde\phi(\Omega)]\frac{\partial\ }{\partial\nu}\left(\frac{\partial v}{\partial\nu}\right)^2$ for the Steklov problem \eqref{ksgen};

\item $G(v)=(1-\sigma)|D^2v|^2+\sigma(\Delta v)^2+\tau|\nabla v|^2-\lambda_F[\tilde\phi(\Omega)]Kv^2-\lambda_F[\tilde\phi(\Omega)]\frac{\partial(v)^2}{\partial\nu}$ for the Steklov problem \eqref{bpgen},
\end{itemize}
where by $K$ we denote the mean curvature of $\partial\tilde{\phi}(\Omega)$.

\end{thm}

\begin{proof}
For the proof of the first part of the theorem we refer to \cite[Theorem 3.1]{bula2013} (see also \cite{lala2004}). Concerning formula \eqref{derivdsys},
we start by recalling that, for $\phi\in\mathcal{A}_{\Omega}$, we have that the pull-back of the operator $M$ is defined by
$$
M[u][v]=\int_{\Omega}D^2(u\circ\phi^{-1}):D^2(v\circ\phi^{-1})|\det D\phi| dx,
$$
for any $u,v\in V(\Omega)$, and similarly also for $B,L,P,J_i$ and $T_i$. We have
\begin{equation*}
{\rm d|}_{\phi =\tilde\phi}\Lambda_{F,s}[\psi]=
-\lambda_{F}^{s+1}[\tilde \phi]   \binom{|F|-1}{s-1}
\sum_{l\in F}
P_{\tilde \phi}\Big[{\rm d|}_{\phi =\tilde \phi}T_{i,\phi }[\psi][u_l]\Big]\Big[ u_l\Big]
\end{equation*}
for all  $\psi\in C^2( \bar\Omega\, ;{\mathbb{R}}^N) $ (cf.\ \cite[proof of Theorem 3.38]{lala2004}), where $ u_l= v_l\circ \tilde \phi $ for all $l\in F$. Note also that by standard regularity theory (see e.g., \cite[Theorem~2.20]{ggs}) $v_l\in H^{4}(\tilde \phi (\Omega ))$ for all $l\in F$.

By standard calculus we have
\begin{equation*}
P_{\tilde \phi}\Big[{\rm d|}_{\phi =\tilde \phi}T_{i,\phi }[\psi][u_l]\Big]\Big[ u_l\Big]
=d|_{\phi=\tilde{\phi}}J_{i,\phi}[\psi][u_l][u_l]\\
 -\lambda_F^{-1}[\tilde\phi]d|_{\phi=\tilde{\phi}}P_{\phi}[\psi][u_l][u_l].
\end{equation*}
Applying Lemmas \ref{fondamentale} and \ref{fondamentalej}, we obtain formula \eqref{derivdsys}.
\end{proof}

We note that formula \eqref{derivdsys} for the Dirichlet problem was already obtained in \cite{orzua} using different techniques, but only for simple eigenvalues. We also remark that some of the specific cases of formula \eqref{derivdsys} were already obtained in \cite{bula2013, buosohinged,bulaproc,bupro,buproimse}. In particular, in \cite{buosohinged} it was derived a different formula for the Navier problem, which was later shown to be equivalent to the one proposed here (see \cite{buosotesi, bulaproc}).

	
	\section{Shape optimization and isovolumetric perturbations}
\label{critico}

Important shape optimization problems for the eigenvalues of elliptic operators were already addressed in \cite{rayleigh}, where the author claims that among all bounded domains in $\mathbb R^2$ of given area, the ball minimizes the first eigenvalue of the Dirichlet Laplacian and Bilaplacian, namely
\begin{equation}\label{fab}
\lambda_1(\Omega)\ge\lambda_1(\Omega^*),
\end{equation}
where $\Omega^*$ is a ball with $|\Omega|=|\Omega^*|$. He does not provide any proof of inequality \eqref{fab}, claiming it trivial for physical evidence. The actual proof of inequality \eqref{fab}, in the case of the Dirichlet Laplacian, is due to Faber \cite{faber} and Krahn \cite{krahn}, and was then followed by similar inequalities for the eigenvalues of the Laplace operator subject to other boundary conditions. We refer to \cite{henrot} for an extensive discussion on the topic.

On the other hand, in the case of the biharmonic operator inequality \eqref{fab} is instead still an open problem and is known as the Rayleigh conjecture. It has been proved only in low dimension, namely $N=2,3$, by Nadirashvili \cite{nadir} and Ashbaugh and Benguria \cite{ash} improving an argument due to Talenti \cite{talenti}. Unfortunately, such an argument does not seem to work in higher dimension. Regarding other boundary conditions, similar inequalities have been derived  for the Neumann problem \eqref{biha}, \eqref{neu} and for the Steklov problem \eqref{bp}, also in quantitative form (see \cite{buchapro,bupro, buproimse,chasman11,chasman15}), showing that in such cases the ball is actually the maximizer. The Steklov problem \eqref{ks} is instead the only one in which the ball has been shown not to be the optimizer for the first eigenvalue (cf.\ \cite{kuttler72}).

Here we consider the following shape optimization problems for the symmetric functions of the eigenvalues
\begin{equation}\label{min}
\min_{V[\phi ]={\rm const}}\Lambda_{F,s}[\phi ]\ \ \ {\rm or}\ \ \ \max_{   V[\phi ]={\rm const}}\Lambda_{F,s}[\phi],
\end{equation}
where $V$ is the real valued function defined on ${\mathcal{A}}_{\Omega }$ which takes $\phi \in {\mathcal{A}}_{\Omega }$ to $V[\phi ]=|\phi (\Omega )|$.
Note that if $\tilde \phi \in {\mathcal{A}}_{\Omega }$ is a minimizer or maximizer in (\ref{min}) then $\tilde \phi $ is a critical domain transformation for the map $\phi\mapsto  \Lambda_{F,s}[\phi ]$ subject to volume constraint, i.e.,
\begin{equation*}
{\rm Ker\ d}|_{\phi =\tilde \phi }V \subset {\rm Ker\ d}|_{\phi =\tilde \phi } \Lambda_{F,s}.
\end{equation*}

The following theorem provides a characterization of all critical domain transformations $\phi$ (see also \cite{bula2013, buosohinged, bulareis, bupro, lalacri}).

\begin{thm}
Let $\Omega $ be a bounded open set in ${\mathbb{R}}^N$ of class $C^1$,  and $F$ be a finite subset of ${\mathbb{N}}$.
Assume that $\tilde \phi\in \Theta_{F, \Omega }$ is such that $\tilde \phi (\Omega )$ is of class $C^{4}$ and that the eigenvalues $\lambda_j[\tilde \phi ]$ have the common value $\lambda_F[\tilde \phi ]$ for all $j\in F$.  Let $\{  v_l\}_{l\in F}$ be an orthornormal basis in $V(\tilde \phi (\Omega ))$ of the eigenspace corresponding to $\lambda_F[\tilde \phi ]$ (the orthonormality being taken with respect to \eqref{prodotto}). Then  $\tilde \phi$
is a critical domain transformation  for any of the functions $\Lambda_{F,s}$, $s=1,\dots , |F|$,  with  volume constraint
if and only if  there exists $c\in {\mathbb{R}}$ such that
	\begin{equation}
	\label{lacondizionedsys}
	\sum_{l=1}^{|F|}G(v_l)=c,\ {\rm\ on\ }\partial\tilde\phi(\Omega).
	\end{equation}
\end{thm}
\begin{proof}
The proof is a straightforward application of Lagrange Multipliers Theorem combined with formula (\ref{derivdsys}).
\end{proof}

	As we have said, balls play a relevant role in the shape optimization of the eigenvalues of the Laplace and biharmonic operators. Hence we need to analyze in more details the behavior of the eigenfunctions on balls. We have the following result (cf.\ \cite[Lemma 4.22]{bupro}).
		
	\begin{thm}
	\label{propedeutico}
		Let $B$ be a ball in $\mathbb{R}^N$ centered at zero, and let $\lambda$ be an eigenvalue of problem (\ref{plateweak}) in $B$. Let $F$ be the subset of
		$\mathbb{N}$ of all $j$ such that the $j$-th eigenvalue of problem (\ref{plateweak}) in $B$ coincides with $\lambda$. Let $v_1,\dots,v_{|F|}$
		be an orthonormal basis of the eigenspace associated with the eigenvalue $\lambda$, where the orthonormality is taken with respect to the scalar
		product in $V(B)$. Then
		$$
		\sum_{j=1}^{|F|}v_j^2,\ \sum_{j=1}^{|F|}|\nabla v_j|^2,\ \sum_{j=1}^{|F|}|\Delta v_j|^2,\ \sum_{j=1}^{|F|}|D^2v_j|^2
		$$ 
	are radial functions.
	\end{thm}	
	\begin{proof}
	First of all, note that by standard regularity theory (see e.g., \cite{adn,ggs}), the functions $v_{j}\in C^{\infty}(\overline B)$ for all $j\in F$.
	
		Let $O_N(\mathbb{R})$ denote the group of orthogonal linear transformations in $\mathbb{R}^N$. Since the operators $P$ and $J_i$, $i=1,2,3$ are invariant under rotations, then
		$v_k\circ R$, where $R\in O_N(\mathbb R)$, is still an eigenfunction with eigenvalue $\lambda$; moreover, $\{v_j\circ R:j=1,\dots, |F|\}$ is another orthonormal basis
		for the eigenspace associate with $\lambda$. Since both $\{v_j:j=1,\dots, |F|\}$ and $\{v_j\circ R:j=1,\dots, |F|\}$
		are orthonormal bases, then there exists $A[R]\in O_N(\mathbb{R})$ with matrix $(A_{ij}[R])_{i,j=1,\dots,|F|}$ such that
		\begin{equation}
		\label{eq}
		v_j=\sum_{l=1}^{|F|}A_{jl}[R]v_l\circ R.
		\end{equation}
		This implies that
		\begin{equation*}
		\sum_{j=1}^{|F|}v_j^2=\sum_{j=1}^{|F|}(v_j\circ R)^2,
		\end{equation*}
		from which we get that $\sum_{j=1}^{|F|}v_j^2$ is radial.
		Moreover, using standard calculus and (\ref{eq}), we get
		\begin{equation*}
		\sum_{j=1}^{|F|}|\nabla v_j|^2=\sum_{j,l_1,l_2=1}^{|F|}A_{jl_1}[R]A_{jl_2}[R]\left(\nabla v_{l_1}\circ R\right)\cdot\left(\nabla v_{l_2}\circ R\right)
		=\sum_{l=1}^{|F|}|\nabla v_l\circ R|^2.
		\end{equation*}
		Similarly,
		\begin{equation*}
		\sum_{j=1}^{|F|}|\Delta v_j|^2=\sum_{j=1}^{|F|}|\Delta v_j\circ R|^2.
		\end{equation*}
		On the other hand,
		\begin{multline*}
		D^2v_j\cdot D^2v_j
		=\sum_{l_1,l_2=1}^{|F|}A_{jl_1}[R]A_{jl_2}[R]R^t\cdot (D^2v_{l_1}\circ R)\cdot R\cdot R^t\cdot (D^2v_{l_2}\circ R)\cdot R\\
		=\sum_{l_1,l_2=1}^{|F|}A_{jl_1}[R]A_{jl_2}[R]R^t\cdot (D^2v_{l_1}\circ R)\cdot (D^2v_{l_2}\circ R)\cdot R,
		\end{multline*}
		therefore
		\begin{equation*}
		|D^2v_j|^2=\mathrm{tr}(D^2v_j\cdot D^2v_j)=\sum_{l_1,l_2=1}^{|F|}A_{jl_1}[R]A_{jl_2}[R](D^2v_{l_1}\circ R):(D^2v_{l_2}\circ R),
		\end{equation*}
		from which we get
		\begin{equation*}
		\sum_{j=1}^{|F|}|D^2v_j|^2=\sum_{j=1}^{|F|}|D^2v_j\circ R|^2.
		\end{equation*}
	\end{proof}

For rotation invariant operators such as the Laplace or the biharmonic operator, it is easy to show that any eigenfunction associated with a simple eigenvalue is radial. Theorem \ref{propedeutico} tells us that, when dealing with a multiple eigenvalue, we cannot consider the eigenfunctions alone, but we have to consider the whole eigenspace. In particular, this is very useful when coupled with condition \eqref{lacondizionedsys}.

	\begin{thm}
	\label{lepallesys}
	Let $\Omega$ be a domain in $\mathbb{R}^N$ of class $C^1$. Let $\tilde{\phi}\in\mathcal{A}_{\Omega}$ be such that
	$\tilde{\phi}(\Omega)$ is a ball. Let $\tilde{\lambda}$ be an eigenvalue of problem (\ref{plateweak}) in $\tilde{\phi}(\Omega)$,
	and let $F$ be the set of $j\in\mathbb{N}$ such that $\lambda_j[\tilde{\phi}]=\tilde{\lambda}$.
	Then $\tilde{\phi}$ is a critical point $\Lambda_{F,s}$  under volume constraint,
	for all $s=1,\dots,|F|$.
	\end{thm}
	
	\begin{proof}
	Thanks to Theorem \ref{propedeutico}, it remains to prove that, for problems \eqref{biha}, \eqref{nav} and \eqref{ksgen}, the function
	$$
	\sum_{j=1}^{|F|}\frac{\partial v_j}{\partial\nu}\left(\frac{\partial\Delta v_j}{\partial\nu}+(1-\sigma){\rm div}_{\partial\tilde{\phi}(\Omega)}(\nu\cdot D^2v_j)_{\partial\tilde{\phi}(\Omega)}\right)
	$$
	is a radial function. In particular, here $V(\tilde{\phi}(\Omega))=H^2(\tilde{\phi}(\Omega))\cap H^1_0(\tilde{\phi}(\Omega))$, hence
	$$
	\frac{\partial\ }{\partial\nu}\nabla v_j=\nabla \frac{\partial v_j}{\partial\nu},
	$$
	from which we get
	$$
	{\rm div}_{\partial\tilde{\phi}(\Omega)}(\nu\cdot D^2v_j)_{\partial\tilde{\phi}(\Omega)}=\Delta_{\partial\tilde{\phi}(\Omega)}\left(\frac{\partial v_j}{\partial\nu}\right)
	$$
	on $\partial\tilde{\phi}(\Omega)$. Therefore
	\begin{multline*}
	\sum_{j=1}^{|F|}\frac{\partial v_j}{\partial\nu}{\rm div}_{\partial\tilde{\phi}(\Omega)}(\nu\cdot D^2v_j)_{\partial\tilde{\phi}(\Omega)}
	=\sum_{j=1}^{|F|}\frac{\partial v_j}{\partial\nu}\Delta_{\partial\tilde{\phi}(\Omega)}\left(\frac{\partial v_j}{\partial\nu}\right)\\
	=\frac 1 2 \Delta_{\partial\tilde{\phi}(\Omega)}\left(\sum_{j=1}^{|F|}\left(\frac{\partial v_j}{\partial\nu}\right)^2\right)
	-\sum_{j=1}^{|F|}\left|\nabla_{\partial\tilde{\phi}(\Omega)} \frac{\partial v_j}{\partial\nu}\right|^2,
	\end{multline*}
	where the two summands on the right-hand side can be shown to be constant on $\partial\tilde{\phi}(\Omega)$ following the lines of the proof of Theorem \ref{propedeutico}.
	
	On the other hand,
	\begin{equation*}
	\sum_{j=1}^{|F|}\nabla v_j\cdot\nabla\Delta v_j=\frac 1 8 \Delta^2\left(\sum_{j=1}^{|F|}v_j^2\right)-\frac{\lambda}4\sum_{j=1}^{|F|}v_j^2
	-\frac 1 4 \sum_{j=1}^{|F|}(\Delta v_j)^2-\frac 1 2  \sum_{j=1}^{|F|}|D^2v_j|^2,
	\end{equation*}
	from which we deduce that $\sum_{j=1}^{|F|}\frac{\partial v_j}{\partial\nu}\frac{\partial\Delta v_j}{\partial\nu}$ is constant on $\partial\tilde{\phi}(\Omega)$.
	\end{proof}

In general, balls are expected to be the extremizer of problems of the type \eqref{min} only when the first eigenvalue is involved (see e.g., \cite{henrot}), and even in this case problem \eqref{ks} shows up as a counterexample. Nevertheless, thanks to Theorem \ref{lepallesys}, we know that balls still are critical domains for all the eigenvalues. It also would be interesting to characterize the family of open sets $\tilde{\phi}(\Omega)$ such that condition \eqref{lacondizionedsys} is satisfied. The only result in this direction is due to Dalmasso \cite{dalmasso}, who proved that the ball is the only domain satisfying condition \eqref{lacondizionedsys} for the first eigenvalue of the biharmonic operator subject to Dirichlet boundary conditions under the additional hypotesis that the first eigenfunction does not change sign.


\section{Some technical lemmas}
\label{tech}

In this section we prove two lemmas that has been used in the proof of Theorem \ref{duesettesys}.
	
	\begin{lemm}
	\label{fondamentale}
	Let $\Omega$ be a bounded domain in $\mathbb{R}^N$ of class $C^1$, and let $\tilde\phi\in\mathcal{A}^2_{\Omega}$ be such that $\tilde\phi(\Omega)$ is of class $C^2$.
	Let $u_1,u_2\in H^2(\Omega)$ be such that $v_1=u_1\circ\tilde{\phi}^{-1}$, $v_2=u_2\circ\tilde{\phi}^{-1}\in H^{4}(\tilde{\phi}(\Omega))$. Then
	\begin{multline}
	\label{formulafondamentale}
	d|_{\phi=\tilde{\phi}}M_{\phi}[\psi][u_1][u_2]=
	\int_{\partial\tilde{\phi}(\Omega)}(D^2v_1:D^2v_2)\zeta\cdot\nu d\sigma\\
+\int_{\partial\tilde{\phi}(\Omega)}\left(\mathrm{div}_{\partial\tilde{\phi}(\Omega)}(\nu\cdot D^2v_1)_{\partial\tilde{\phi}(\Omega)}\nabla v_2+\mathrm{div}_{\partial\tilde{\phi}(\Omega)}(\nu\cdot D^2v_2)_{\partial\tilde{\phi}(\Omega)}\nabla v_1\right)\cdot\zeta d\sigma\\
+\int_{\partial\tilde{\phi}(\Omega)}\left(\frac{\partial\Delta v_1}{\partial\nu}\nabla v_2+\frac{\partial\Delta v_2}{\partial\nu}\nabla v_1\right)\cdot\zeta d\sigma
-\int_{\tilde{\phi}(\Omega)}\left(\Delta^2v_1\nabla v_2+\Delta^2v_2\nabla v_1\right)\cdot\zeta d\sigma\\
-\int_{\partial\tilde{\phi}(\Omega)}\left(\frac{\partial^2v_1}{\partial\nu^2}\nabla v_2+\frac{\partial^2v_2}{\partial\nu^2}\nabla v_1\right)\cdot\frac{\partial\zeta}{\partial\nu}d\sigma\\
-\int_{\partial\tilde{\phi}(\Omega)}\left(\frac{\partial^2v_1}{\partial\nu^2}\frac{\partial\ }{\partial\nu}\nabla v_2+\frac{\partial^2v_2}{\partial\nu^2}\frac{\partial\ }{\partial\nu}\nabla v_1\right)\cdot\zeta d\sigma,
	\end{multline}
	\begin{multline}
	\label{formulafondamentale2}
	d|_{\phi=\tilde{\phi}}B_{\phi}[\psi][u_1][u_2]=
	\int_{\partial\tilde{\phi}(\Omega)}\Delta v_1\Delta v_2\zeta\cdot\nu d\sigma\\
+\int_{\partial\tilde{\phi}(\Omega)}\left(\frac{\partial\Delta v_1}{\partial\nu}\nabla v_2+\frac{\partial\Delta v_2}{\partial\nu}\nabla v_1\right)\cdot\zeta d\sigma
-\int_{\tilde{\phi}(\Omega)}\left(\Delta^2v_1\nabla v_2+\Delta^2v_2\nabla v_1\right)\cdot\zeta d\sigma\\
-\int_{\partial\tilde{\phi}(\Omega)}\left(\Delta v_1\nabla v_2+\Delta v_2\nabla v_1\right)\cdot\frac{\partial\zeta}{\partial\nu}d\sigma\\
-\int_{\partial\tilde{\phi}(\Omega)}\left(\Delta v_1\frac{\partial\ }{\partial\nu}\nabla v_2+\Delta v_2\frac{\partial\ }{\partial\nu}\nabla v_1\right)\cdot\zeta d\sigma,
	\end{multline}
and
	\begin{multline}
	\label{formulafondamentale3}
	d|_{\phi=\tilde{\phi}}L_{\phi}[\psi][u_1][u_2]=
	-\int_{\partial\tilde{\phi}(\Omega)}\nabla v_1\cdot\nabla v_2\zeta\cdot\nu d\sigma\\
+\int_{\partial\tilde{\phi}(\Omega)}\left(\frac{\partial v_1}{\partial\nu}\nabla v_2+\frac{\partial v_2}{\partial\nu}\nabla v_1\right)\cdot\zeta d\sigma
-\int_{\tilde{\phi}(\Omega)}\left(\Delta v_1\nabla v_2+\Delta v_2\nabla v_1\right)\cdot\zeta d\sigma,
	\end{multline}
	for all $\psi\in C^2(\overline{\Omega};\mathbb{R}^N)$, where $\zeta=\psi\circ\tilde{\phi}^{-1}$. 
	\end{lemm}
	
	\begin{proof}
	First of all, we observe that the proof of \eqref{formulafondamentale} and of \eqref{formulafondamentale3} can be done following that of \cite[Lemma 4.4]{bupro} (we also refer to \cite[Lemmas 2.4 and 2.6]{buosotesi}). As for \eqref{formulafondamentale2}, we have (see also \cite[Lemma 2.5]{buosotesi})
	\begin{multline}
	\label{311}
		d|_{\phi=\tilde{\phi}} B_{\phi}[\psi][u_1][u_2]\\
			 =\int_{\Omega}(d|_{\phi=\tilde{\phi}}\Delta(u_1\circ\phi^{-1})\circ\phi)[\psi]
				(\Delta(u_2\circ\tilde{\phi}^{-1})\circ\tilde{\phi})|\det D\tilde{\phi}|dx\  \\
				+\int_{\Omega}(\Delta(u_1\circ\tilde{\phi}^{-1})\circ\tilde{\phi})
				(d|_{\phi=\tilde{\phi}}\Delta(u_2\circ\phi^{-1})\circ\phi)[\psi]|\det D\tilde{\phi}|dx\\
			 +\int_{\Omega}(\Delta(u_1\circ\tilde{\phi}^{-1})\circ\tilde{\phi})
				(\Delta(u_2\circ\tilde{\phi}^{-1})\circ\tilde{\phi})d|_{\phi=\tilde{\phi}}|\det D\phi|[\psi]dx,
				\end{multline}
		and we note that, by the equality
		\begin{equation}
\label{der4}
\left[
\left({\rm d}|_{\phi =\tilde \phi } \left({\mathrm{det}}\nabla\phi\right)[\psi]\right)\circ
\tilde \phi^{(-1)}\right]{\mathrm{det}}\nabla\tilde \phi^{(-1)}=
{\mathrm{div}}\left(\psi\circ\tilde  \phi^{(-1)} \right),
\end{equation}
		the last summand in (\ref{311}) equals
	$$\int_{\tilde{\phi}(\Omega)}\Delta v_1\Delta v_2\mathrm{div}\zeta dy.$$
	We have
	\begin{multline}
	\label{312}
	\int_{\tilde \phi(\Omega)}\frac{\partial^2v_1}{\partial y_r\partial y_s}\frac{\partial\zeta_r}{\partial y_s}\Delta v_2 dy
	=\int_{\partial\tilde \phi(\Omega)}\frac{\partial v_1}{\partial y_s}\frac{\partial\zeta_r}{\partial y_s}\nu_r\Delta v_2 d\sigma\\
	-\int_{\tilde \phi(\Omega)}\frac{\partial v_1}{\partial y_s}\frac{\partial{\rm div}\zeta}{\partial y_s}\Delta v_2 dy
	-\int_{\tilde \phi(\Omega)}\frac{\partial v_1}{\partial y_s}\frac{\partial\zeta_r}{\partial y_s}\frac{\partial\Delta v_2}{\partial y_r}dy\\
	=\int_{\partial\tilde \phi(\Omega)}\frac{\partial v_1}{\partial y_s}\frac{\partial\zeta_r}{\partial y_s}\nu_r\Delta v_2 d\sigma
	-\int_{\tilde \phi(\Omega)}\frac{\partial v_1}{\partial y_s}\frac{\partial\zeta_r}{\partial y_s}\frac{\partial\Delta v_2}{\partial y_r}dy\\
	-\int_{\partial\tilde \phi(\Omega)}\frac{\partial v_1}{\partial \nu}\Delta v_2 {\rm div}\zeta d\sigma
+\int_{\tilde \phi(\Omega)}\Delta v_1\Delta v_2 {\rm div}\zeta dy\\
+\int_{\tilde \phi(\Omega)}\nabla v_1\cdot\nabla\Delta v_2 {\rm div}\zeta dy,
	\end{multline}
	and
	\begin{multline}
	\label{313}
	\int_{\tilde \phi(\Omega)}\frac{\partial v_1}{\partial y_s}\Delta\zeta_s\Delta v_2 dy
	=\int_{\partial\tilde \phi(\Omega)}\frac{\partial v_1}{\partial y_s}\frac{\partial\zeta_s}{\partial\nu}\Delta v_2 d\sigma\\
	-\int_{\tilde \phi(\Omega)}\frac{\partial^2 v_1}{\partial y_i\partial y_s}\frac{\partial\zeta_s}{\partial y_i}\Delta v_2 dy
	-\int_{\tilde \phi(\Omega)}\frac{\partial v_1}{\partial y_s}\frac{\partial\zeta_s}{\partial y_i}\frac{\partial\Delta v_2 }{\partial y_i}dy\\
	=\int_{\partial\tilde \phi(\Omega)}\frac{\partial v_1}{\partial y_s}\frac{\partial\zeta_s}{\partial\nu}\Delta v_2 d\sigma
	-\int_{\tilde \phi(\Omega)}\frac{\partial v_1}{\partial y_s}\frac{\partial\zeta_s}{\partial y_i}\frac{\partial\Delta v_2 }{\partial y_i}dy\\
	-\int_{\partial\tilde \phi(\Omega)}\frac{\partial v_1}{\partial y_i}\frac{\partial\zeta_s}{\partial y_i}\nu_s\Delta v_2 d\sigma
	+\int_{\tilde \phi(\Omega)}\frac{\partial v_1}{\partial y_i}\frac{\partial{\rm div}\zeta}{\partial y_i}\Delta v_2 dy\\
	+\int_{\tilde \phi(\Omega)}\frac{\partial v_1}{\partial y_i}\frac{\partial\zeta_s}{\partial y_i}\frac{\partial\Delta v_2 }{\partial y_s}dy
	=\int_{\partial\tilde \phi(\Omega)}\frac{\partial v_1}{\partial y_s}\frac{\partial\zeta_s}{\partial\nu}\Delta v_2 d\sigma\\
	-\int_{\tilde \phi(\Omega)}\frac{\partial v_1}{\partial y_s}\frac{\partial\zeta_s}{\partial y_i}\frac{\partial\Delta v_2 }{\partial y_i}dy
	-\int_{\partial\tilde \phi(\Omega)}\frac{\partial v_1}{\partial y_i}\frac{\partial\zeta_s}{\partial y_i}\nu_s\Delta v_2 d\sigma\\
	+\int_{\tilde \phi(\Omega)}\frac{\partial v_1}{\partial y_i}\frac{\partial\zeta_s}{\partial y_i}\frac{\partial\Delta v_2 }{\partial y_s}dy
	+\int_{\partial\tilde \phi(\Omega)}\frac{\partial v_1}{\partial \nu}\Delta v_2{\rm div}\zeta d\sigma\\
	-\int_{\tilde \phi(\Omega)}\Delta v_1\Delta v_2{\rm div}\zeta dy
	-\int_{\tilde \phi(\Omega)}\nabla v_1\cdot\nabla\Delta v_2 {\rm div}\zeta dy.
	\end{multline}

Combining (\ref{311}), \eqref{312}, and (\ref{313}) we get
\begin{multline}
\label{314}
d|_{\phi=\tilde{\phi}} B_{\phi}[\psi][u_1][u_2]
=-\int_{\partial\tilde\phi(\Omega)}\left(\frac{\partial v_1}{\partial y_s}\Delta v_2+\frac{\partial v_2}{\partial y_s}\Delta v_1\right)\frac{\partial\zeta_r}{\partial y_s}\nu_r d\sigma\\
+\int_{\tilde\phi(\Omega)}\left(\frac{\partial v_1}{\partial y_s}\frac{\partial\Delta v_2}{\partial y_r}+\frac{\partial v_2}{\partial y_s}\frac{\partial\Delta v_1}{\partial y_r}\right)\frac{\partial\zeta_r}{\partial y_s}dy\\
+\int_{\partial\tilde\phi(\Omega)}\left(\frac{\partial v_1}{\partial\nu}\Delta v_2+\frac{\partial v_2}{\partial\nu}\Delta v_1\right){\rm div}\zeta d\sigma\\
-\int_{\tilde\phi(\Omega)}\Delta v_1\Delta v_2 {\rm div}\zeta dy
-\int_{\tilde\phi(\Omega)}\left(\nabla v_1\cdot\nabla\Delta v_2+\nabla v_2\cdot\nabla\Delta v_1\right){\rm div}\zeta dy\\
-\int_{\partial\tilde\phi(\Omega)}\left(\frac{\partial v_1}{\partial y_s}\Delta v_2+\frac{\partial v_2}{\partial y_s}\Delta v_1\right)\frac{\partial\zeta_s}{\partial\nu} d\sigma\\
+\int_{\tilde\phi(\Omega)}\left(\frac{\partial v_1}{\partial y_s}\frac{\partial\Delta v_2}{\partial y_i}+\frac{\partial v_2}{\partial y_s}\frac{\partial\Delta v_1}{\partial y_i}\right)\frac{\partial\zeta_s}{\partial y_i}dy.
\end{multline}
The last summand in the right-hand side of (\ref{314}) equals
\begin{multline*}
\int_{\partial\tilde\phi(\Omega)}\left(\frac{\partial\Delta v_1}{\partial \nu}\nabla v_2+\frac{\partial\Delta v_2}{\partial \nu}\nabla v_1\right)\cdot\zeta d\sigma
-\int_{\tilde\phi(\Omega)}\left(\Delta^2 v_1\nabla v_2+\Delta^2v_2\nabla v_1\right)\cdot\zeta dy\\
-\int_{\tilde\phi(\Omega)}\left(\frac{\partial^2 v_1}{\partial y_i\partial y_s}\frac{\partial\Delta v_2}{\partial y_i}+\frac{\partial^2 v_2}{\partial y_i\partial y_s}\frac{\partial\Delta v_1}{\partial y_i}\right)\zeta_sdy\\
=\int_{\partial\tilde\phi(\Omega)}\left(\frac{\partial\Delta v_1}{\partial \nu}\nabla v_2+\frac{\partial\Delta v_2}{\partial \nu}\nabla v_1\right)\cdot\zeta d\sigma
-\int_{\tilde\phi(\Omega)}\left(\Delta^2 v_1\nabla v_2+\Delta^2v_2\nabla v_1\right)\cdot\zeta dy\\
-\int_{\partial\tilde\phi(\Omega)}\left(\nabla v_1\cdot\nabla\Delta v_2+\nabla v_2\cdot\nabla\Delta v_1\right)\zeta\cdot\nu d\sigma\\
+\int_{\tilde\phi(\Omega)}\left(\nabla v_1\cdot\nabla\Delta v_2+\nabla v_2\cdot\nabla\Delta v_1\right){\rm div}\zeta dy\\
+\int_{\tilde\phi(\Omega)}\left(\frac{\partial v_1}{\partial y_i}\frac{\partial^2\Delta v_2}{\partial y_i\partial y_s}+\frac{\partial v_2}{\partial y_i}\frac{\partial^2\Delta v_1}{\partial y_i\partial y_s}\right)\zeta_sdy,
\end{multline*}
while the second one equals
\begin{multline*}
\int_{\partial\tilde\phi(\Omega)}\left(\frac{\partial v_1}{\partial\nu}\nabla\Delta v_2+\frac{\partial v_2}{\partial\nu}\nabla\Delta v_1\right)\cdot\zeta d\sigma\\
-\int_{\tilde\phi(\Omega)}\left(\frac{\partial v_1}{\partial y_s}\frac{\partial^2\Delta v_2}{\partial y_r\partial y_s}+\frac{\partial v_2}{\partial y_s}\frac{\partial^2\Delta v_1}{\partial y_r\partial y_s}\right)\zeta_rdy\\
-\int_{\partial\tilde\phi(\Omega)}\Delta v_1\Delta v_2\zeta\cdot\nu d\sigma
+\int_{\tilde\phi(\Omega)}\Delta v_1\Delta v_2{\rm div}\zeta dy.
\end{multline*}
Hence we have
\begin{multline}
\label{317}
d|_{\phi=\tilde{\phi}} B_{\phi}[\psi][u_1][u_2]
=-\int_{\partial\tilde\phi(\Omega)}\left(\frac{\partial v_1}{\partial y_s}\Delta v_2+\frac{\partial v_2}{\partial y_s}\Delta v_1\right)\frac{\partial\zeta_r}{\partial y_s}\nu_r d\sigma\\
+\int_{\partial\tilde\phi(\Omega)}\left(\frac{\partial v_1}{\partial\nu}\nabla\Delta v_2+\frac{\partial v_2}{\partial\nu}\nabla\Delta v_1\right)\cdot\zeta d\sigma\\
+\int_{\partial\tilde\phi(\Omega)}\left(\frac{\partial v_1}{\partial\nu}\Delta v_2+\frac{\partial v_2}{\partial\nu}\Delta v_1\right){\rm div}\zeta d\sigma\\
-\int_{\partial\tilde\phi(\Omega)}\Delta v_1\Delta v_2\zeta\cdot\nu d\sigma
-\int_{\partial\tilde\phi(\Omega)}\left(\frac{\partial v_1}{\partial y_s}\Delta v_2+\frac{\partial v_2}{\partial y_s}\Delta v_1\right)\frac{\partial\zeta_s}{\partial\nu} d\sigma\\
+\int_{\partial\tilde\phi(\Omega)}\left(\frac{\partial\Delta v_1}{\partial \nu}\nabla v_2+\frac{\partial\Delta v_2}{\partial \nu}\nabla v_1\right)\cdot\zeta d\sigma\\
-\int_{\tilde\phi(\Omega)}\left(\Delta^2 v_1\nabla v_2+\Delta^2v_2\nabla v_1\right)\cdot\zeta dy\\
-\int_{\partial\tilde\phi(\Omega)}\left(\nabla v_1\cdot\nabla\Delta v_2+\nabla v_2\cdot\nabla\Delta v_1\right)\zeta\cdot\nu d\sigma.
\end{multline}
The first summand in (\ref{317}) equals
\begin{multline*}
-\int_{\partial\tilde\phi(\Omega)}\left(\frac{\partial v_1}{\partial \nu}\Delta v_2+\frac{\partial v_2}{\partial \nu}\Delta v_1\right)\frac{\partial\zeta_r}{\partial \nu}\nu_r d\sigma\\
+\int_{\partial\tilde\phi(\Omega)}\left(\nabla_{\partial\tilde\phi(\Omega)}v_1\cdot\nabla_{\partial\tilde\phi(\Omega)}\Delta v_2+\nabla_{\partial\tilde\phi(\Omega)}v_2\cdot\nabla_{\partial\tilde\phi(\Omega)}\Delta v_1\right)\zeta\cdot\nu d\sigma\\
+\int_{\partial\tilde\phi(\Omega)}\left(\Delta v_1\Delta_{\partial\tilde\phi(\Omega)}v_2+\Delta v_2\Delta_{\partial\tilde\phi(\Omega)}v_1\right)\zeta\cdot\nu d\sigma\\
+\int_{\partial\tilde\phi(\Omega)}\left(\Delta v_1\nabla_{\partial\tilde\phi(\Omega)}v_2+\Delta v_1\nabla_{\partial\tilde\phi(\Omega)}v_1\right)\cdot(\nabla_{\partial\tilde\phi(\Omega)}\nu_r)\zeta_rd\sigma,
\end{multline*}
while the second one equals
\begin{multline*}
\int_{\partial\tilde\phi(\Omega)}\left(\frac{\partial v_1}{\partial\nu}\frac{\partial\Delta v_2}{\partial\nu}+\frac{\partial v_2}{\partial\nu}\frac{\partial\Delta v_1}{\partial\nu}\right)\zeta\cdot\nu d\sigma\\
+\int_{\partial\tilde\phi(\Omega)}K\left(\frac{\partial v_1}{\partial\nu}\Delta v_2+\frac{\partial v_2}{\partial\nu}\Delta v_1\right)\zeta\cdot\nu d\sigma\\
-\int_{\partial\tilde\phi(\Omega)}\left(\frac{\partial v_1}{\partial\nu}\Delta v_2+\frac{\partial v_2}{\partial\nu}\Delta v_1\right){\rm div}_{\partial\tilde\phi(\Omega)}\zeta d\sigma\\
-\int_{\partial\tilde\phi(\Omega)}\left(\Delta v_1\nabla_{\partial\tilde\phi(\Omega)}\frac{\partial v_2}{\partial\nu}+\Delta v_2\nabla_{\partial\tilde\phi(\Omega)}\frac{\partial v_1}{\partial\nu}\right)\cdot\zeta d\sigma,
\end{multline*}
where $K$ denotes the mean curvature of $\partial\tilde\phi(\Omega)$. Therefore the first three terms in the right-hand side of (\ref{317}) equal
\begin{multline}
\label{320}
\int_{\partial\tilde\phi(\Omega)}\left(\nabla v_1\cdot\nabla\Delta v_2+\nabla v_2\cdot\nabla\Delta v_1\right)\zeta\cdot\nu d\sigma\\
-2\int_{\partial\tilde\phi(\Omega)}\Delta v_1\Delta v_2\zeta\cdot\nu d\sigma
-\int_{\partial\tilde\phi(\Omega)}\left(\Delta v_1\frac{\partial^2v_2}{\partial\nu^2}+\Delta v_2\frac{\partial^2v_1}{\partial\nu^2}\right)\zeta\cdot\nu d\sigma\\
+\int_{\partial\tilde\phi(\Omega)}\left(\Delta v_1\nabla_{\partial\tilde\phi(\Omega)}v_2+\Delta v_1\nabla_{\partial\tilde\phi(\Omega)}v_1\right)\cdot(\nabla_{\partial\tilde\phi(\Omega)}\nu_r)\zeta_rd\sigma\\
-\int_{\partial\tilde\phi(\Omega)}\left(\Delta v_1\nabla_{\partial\tilde\phi(\Omega)}\frac{\partial v_2}{\partial\nu}+\Delta v_2\nabla_{\partial\tilde\phi(\Omega)}\frac{\partial v_1}{\partial\nu}\right)\cdot\zeta d\sigma.
\end{multline}
Now note that summing the third and the fifth terms in (\ref{320}) we get
\begin{multline}
\label{321}
-\int_{\partial\tilde\phi(\Omega)}\left(\Delta v_1\nabla\frac{\partial v_2}{\partial\nu}+\Delta v_2\nabla\frac{\partial v_1}{\partial\nu}\right)\cdot\zeta d\sigma\\
=-\int_{\partial\tilde\phi(\Omega)}\left(\Delta v_1\frac{\partial\ }{\partial\nu}\nabla v_2+\Delta v_2\frac{\partial\ }{\partial\nu}\nabla v_1\right)\cdot\zeta d\sigma\\
-\int_{\partial\tilde\phi(\Omega)}\left(\Delta v_1\nabla_{\partial\tilde\phi(\Omega)}v_2+\Delta v_2\nabla_{\partial\tilde\phi(\Omega)}v_1\right)\cdot(\nabla_{\partial\tilde\phi(\Omega)}\nu_r)\zeta_r d\sigma.
\end{multline}
Using (\ref{317}), (\ref{320}) and (\ref{321}), we finally get formula (\ref{formulafondamentale2}).
\end{proof}

\begin{lemm}
\label{fondamentalej}
Let $\Omega$ be a bounded domain in $\mathbb{R}^N$ of class $C^1$, and let $\tilde\phi\in\mathcal{A}^2_{\Omega}$ be such that $\tilde\phi(\Omega)$ is of class $C^2$.
Let $u_1,u_2\in H^2(\Omega)$ be such that $v_1=u_1\circ\tilde{\phi}^{-1}$, $v_2=u_2\circ\tilde{\phi}^{-1}\in H^{4}(\tilde{\phi}(\Omega))$. Then
\begin{equation}
	\label{jay1}
			d|_{\phi=\tilde{\phi}}J_{1,\phi}[\psi][u_1][u_2]
			=\int_{\tilde{\phi}(\Omega)}v_1v_2\mathrm{div}\zeta dy,
	\end{equation}
	
	\begin{multline}
	\label{jay2}
	d|_{\phi=\tilde{\phi}}J_{2,\phi}[\psi][u_1][u_2]
			=\int_{\partial\tilde{\phi}(\Omega)}\left(K\frac{\partial v_1}{\partial\nu}\frac{\partial v_2}{\partial\nu}+\frac{\partial\ }{\partial\nu}(\frac{\partial v_1}{\partial\nu}\frac{\partial v_2}{\partial\nu})\right)\zeta\cdot\nu d\sigma\\
	-\int_{\partial\tilde{\phi}(\Omega)}\nabla(\frac{\partial v_1}{\partial\nu}\frac{\partial v_2}{\partial\nu})\cdot\mu d\sigma
	-2\int_{\partial\tilde{\phi}(\Omega)}\nabla(\frac{\partial v_1}{\partial\nu}\frac{\partial v_2}{\partial\nu})\frac{\partial\zeta}{\partial\nu}\cdot\nu d \sigma,
	\end{multline}
	and
	\begin{equation}\label{jay3}
			d|_{\phi=\tilde{\phi}}J_{3,\phi}[\psi][u_1][u_2]
			=\int_{\partial\tilde{\phi}(\Omega)}\left(Kv_1v_2+\frac{\partial\ }{\partial\nu}(v_1v_2)\right)\zeta\cdot\nu d\sigma
			-\int_{\partial\tilde{\phi}(\Omega)}\nabla(v_1v_2)\cdot\mu d\sigma,
	\end{equation}
for all $\psi\in C^2(\overline{\Omega};\mathbb{R}^N)$, where $\zeta=\psi\circ\tilde{\phi}^{-1}$ and $K$ is the mean curvature on $\partial\tilde{\phi}(\Omega)$. 
\end{lemm}

\begin{proof}
Formula \eqref{jay1} is immediate from \eqref{der4}, while for formula \eqref{jay3} we refer to \cite[Lemma 3.3]{lambertisteklov}.

Regarding formula \eqref{jay2}, we start observing that
\begin{equation*}
J_2[u_1][u_2]=\int_{\partial\Omega}\nabla u_1\cdot\nabla u_2d\sigma,
\end{equation*}
since $\nabla u=\frac{\partial u}{\partial\nu}\nu$ for any $u\in H^2(\Omega)\cap H^1_0(\Omega)$. Hence
\begin{equation*}
J_{2,\phi}[u_1][u_2]=\int_{\partial\Omega}\left(\nabla u_1\cdot\nabla(\phi^{-1})\right)\cdot\left(\nabla u_2\cdot\nabla(\phi^{-1})\right)|\nu\cdot\nabla(\phi^{-1})||\det \nabla\phi|d\sigma,
\end{equation*}
from which we get
	\begin{multline}\label{uu}
	d|_{\phi=\tilde{\phi}}J_{2,\phi}[\psi][u_1][u_2]
=-\int_{\partial\tilde{\phi}(\Omega)}\frac{\partial v_1}{\partial y_r}\left(\frac{\partial \zeta_r}{\partial y_s}+\frac{\partial \zeta_s}{\partial y_r}\right)\frac{\partial v_2}{\partial y_s} d\sigma
+\int_{\partial\tilde{\phi}(\Omega)}\nabla v_1\cdot\nabla v_2{\rm div\ } \zeta d\sigma\\
-\int_{\partial\tilde{\phi}(\Omega)}\nabla v_1\cdot\nabla v_2 \frac{\partial \zeta}{\partial \nu}\cdot\nu d\sigma
=-2\int_{\partial\tilde{\phi}(\Omega)}\frac{\partial v_1}{\partial \nu}\frac{\partial v_2}{\partial \nu}\frac{\partial \zeta}{\partial \nu}\cdot\nu d\sigma
+\int_{\partial\tilde{\phi}(\Omega)}\frac{\partial v_1}{\partial \nu}\frac{\partial v_2}{\partial \nu}{\rm div}_{\partial\tilde{\phi}(\Omega)}\zeta d\sigma.	
	\end{multline}

Using the Tangential Green's Formula (cf.\ \cite[\S 8.5]{delfour}) we have
\begin{multline}\label{uuu}
\int_{\partial\tilde{\phi}(\Omega)}\frac{\partial v_1}{\partial \nu}\frac{\partial v_2}{\partial \nu}{\rm div}_{\partial\tilde{\phi}(\Omega)}\zeta d\sigma
=\int_{\partial\tilde{\phi}(\Omega)}K\frac{\partial v_1}{\partial \nu}\frac{\partial v_2}{\partial \nu}\zeta\cdot\nu d\sigma\\
-\int_{\partial\tilde{\phi}(\Omega)}\nabla_{\partial\tilde{\phi}(\Omega)}\left(\frac{\partial v_1}{\partial \nu}\frac{\partial v_2}{\partial \nu}\right)\cdot\zeta d\sigma,
\end{multline}
where $\nabla_{\partial\tilde{\phi}(\Omega)}$ is the tangential gradient. Combining \eqref{uu} and \eqref{uuu} we obtain \eqref{jay2}.
\end{proof}

	{\bf Acknowledgements.} The author is very thankful to Prof.\ Pier Domenico Lamberti and Dr.\ Luigi Provenzano for useful comments and discussions. The author has been partially supported by the research project
`Singular perturbation problems for differential operators' Progetto di Ateneo
of the University of Padova, and by the research project FIR (Futuro in Ricerca) 2013 `Geometrical and qualitative aspects of PDE's'.  The author is a  member of the Gruppo Nazionale
per l'Analisi Matematica, la Probabilit\`a e le loro Applicazioni (GNAMPA) of
the Istituto Nazionale di Alta Matematica (INdAM).
	

$ $

\noindent {
Davide Buoso \\
Dipartimento di Scienze Matematiche ``G.L.\ Lagrange''\\
Politecnico di Torino\\
corso Duca degli Abruzzi, 24\\
10129 Torino\\
Italy\\
e-mail:	davide.buoso@polito.it
}

\end{document}